\documentclass[a4paper, 12pt]{amsart}

\usepackage{amsmath, latexsym,  amssymb, amscd}
\usepackage[all]{xy}
\newcommand{\idi}{{\rm Id}}
\newcommand{\defalf}{\alpha}

\newcommand{\mat}{{\rm{Mat}}}
\newcommand{\moltp}{{N -1 \choose n -1}}

\newcommand{\abv}{2}
\newcommand{\cnbundle}{{\mathcal{O}_{N}}}
\newcommand{\cbo}{{\mathcal{O}_{n}}}
\newcommand{\cani}{{\mathcal{O}}}
\newcommand{\chern}{{\rm{c_1}}}
\newcommand{\bertrand}{\kappa}

\newcommand{\invf}{{\hat\phi}}
\newcommand{\piuc}{\boxplus}

\newcommand{\gum}{E}
\newcommand{\ambiente}{\gum^n}

\newcommand{\coefmeta}{{\frac{1}{n-d}+\eta}}
\newcommand{\invcd}{{\frac{1}{n-d}}}

\newcommand{\codv}{n-d}

\newcommand{\invh}{{\frac{1}{n-d}}}

\newcommand{\ef}{\phi}
\newcommand{\efdual}{\hat{\phi}}

\newcommand{\czero}{c_0(E^n, \eta)}
\newcommand{\czerosp}{c_0}
\newcommand{\cuno}{c_1(E^n,\eta)}
\newcommand{\cunosp}{c_1}
\newcommand{\cunospdue}{c_2}
\newcommand{\ch}{c_2}

\newcommand{\elle}{{\mathcal{L}}}

\newcommand{\stab}{{\rm{Stab}}\,\,}

\newcommand{\stabv}{{\rm{Stab}}\,\,V}

\newcommand{\rend}{{\rm End}(E)}

\newcommand{\qe}{\mathbb{Q}}
\newcommand{\co}{\mathbb{C}}
\newcommand{\re}{\mathbb{R}}
\newcommand{\ze}{\mathbb{Z}}

\newtheorem{thm}{Theorem}[section]
\newtheorem{con}[thm]{Conjecture}
\newtheorem{propo}[thm]{Proposition}
\newtheorem{lem}[thm]{Lemma}
\newtheorem{cor}[thm]{Corollary}
\newtheorem{D}[thm]{Definition}

\newtheorem*{thm1}{Theorem \ref{uno}}

\title[{ A functorial  lower bound for the Essential Minimum} ]
{A functorial lower bound for the Essential Minimum  of  varieties in a power of an elliptic curve}

\author[{ Evelina Viada}]{  
 }

\begin{document}

\maketitle
\centerline{Evelina Viada\footnote{Evelina Viada,
     University of Basel, Rheinsprung 21, CH-4051 Basel
Switzerland,
    evelina.viada@unibas.ch.}
    \footnote{Supported by the SNF (Swiss National Science Foundation).}
}

 \begin{abstract}
 
 A subvariety $V$ of an abelian variety is translate if it is  the union of translates of proper algebraic subgroups. An irreducible $V$ is transverse if it is not contained in any translate variety.  Effective sharp lower bounds for a transverse subvarieties of a power of an elliptic curve $E$ are known.
 Here, we prove a sharp lower bound for the essential minimum of  non-translate subvarieties  of $E^g$.  
  \end{abstract}

\section{introduction}

In this work,  variety means a variety defined over the algebraic numbers.  
Let $A$ be an abelian variety. 
To  a symmetric  ample line bundle $\elle$ on $A$ we associate an embedding $i_\elle: A\hookrightarrow \mathbb{P}^m$ defined by the minimal power of $\elle$ which is very ample. Heights and degrees corresponding to $\elle$ are computed  via such an embedding. More precisely,   $h_\elle$  is the $\elle$-canonical N\'eron-Tate height. The degree of a subvariety of $A$ is the degree of its image under $i_\elle$. 
For $E$  an elliptic curve, we denote by $\cani$ the line bundle on $E$ defined by the neutral element. 
 The standard line bundle $\cbo$ on $\gum^n$  is the tensor product of the pull-back of $\cani$ via the natural projections.   In this paper, we consider irreducible subvarieties $V$ of $E^n$. We are going to analyse the sets of points of small height of $V$, with respect to different line bundles.\\

It is well-known that the set of torsion points of $A$ is a dense subset of  $A$ and it coincides with the points of trivial height on $A$.  A quite deep problem is to 
 investigate lower bounds for the points of $A$ of
positive height (ex. 
  Baker and Silvermann  \cite{B}, \cite{Sil}, David and Hindry \cite{sihi}, Masser \cite{Masser}). A general question is to study the height of
points on  an algebraic subvariety $V$ of $A$ of positive dimension $d$. 
We say that  $V$ is  {\it torsion} if it is the union of some
translates of algebraic subgroups of $A$ by torsion points. Moreover an irreducible $V$ is {\em transverse} if it is not contained in any translate of an abelian subvariety.
The Manin-Mumford conjecture (Raynaud \cite{raynaud}) ensures that the
torsion points are dense in $V$ if and only if $V$ is torsion. The Bogomolov Conjecture (Ullmo
\cite{ulmo} and Zhang \cite{zhang}) is a statemnt on the points of $V$
of positive height.\\

Essential minimum: {\it
The essential minimum 
$\mu_\elle (V)$ is the supremum of the reals $\theta$ such that the
set  $ \{ x \in V(\overline{\qe}) \,\,\,:\,\,\,h_\elle(x)\le \theta\}$
is non-dense in $V$.}\\

Bogomolov Conjecture:  {\it The essential minimum $\mu_\elle(V)$ is strictly positive if and only if $V$ is non-torsion.}\\ 

The problem of giving explicit bounds for  $\mu_\elle(V)$ depending on the invariants of $V$ and
of  the  ambient variety has been investigated in several deep works (for instance in the toric case Amoroso and David  \cite{fra}, Bombieri and Zannier \cite{BZ},   Schmidt \cite{Sc} and in the abelian case David and Philippon \cite{sipacommentari}, \cite{sipa}). Different points of view can be assumed.
On one side one can fix,  once  for all, the embedding of  the ambient
variety in a projective space. This corresponds to fixing a polarization on $A$. On the other hand one can vary the embedding  and  investigates the dependence of the essential minimum on the embedding.\\

Most of the known results are obtained after  fixing a standard
embedding.  Amoroso and
David \cite{fra} theorem 1.4 prove a  quasi-optimal lower bound of the
essential minimum  of
 a transverse subvariety of  a torus.  A new and simplified approach
 is introduced by Amoroso and the author \cite{fraio}. David and Philippon  \cite{sipa1}, \cite{sipacommentari}  optained
several non-optimal lower bounds for the essential minimum of a
subvariety of  a
general abelian variety and stronger bounds for a subvariety of a power of an elliptic curve.  
In a preprint, Galateau \cite{Gal} proves a quasi-optimal lower bound in a product of elliptic
curves. 

If the polarization varies,  we expect  that  the polarization  influences the
  essential minimum in a natural way. This kind of problems are called `functorial Bogomolov'. For tori, a functorial conjecture has  been introduced by 
 Amoroso and David \cite{fra}. Not much is proven in this direction.
 Some results in the toric case and for $V$  a translate of a subtorus are  proven by Sombra and Philippon \cite{sombrapa}. 
A weak-functorial result
 has applications in the context of the so called Zilber-Pink
 conjecture, which is a generalization of a work of Bombieri, Masser
 and Zannier and of the Mordell-Lang plus Bogomolov problem.  In our
 works \cite{ant} proposition 13.3, \cite{irmn} proposition 4.3 and \cite{ijnt}  proposition 4.6 we produce a functorial result in the very special case of an isogeny of $E^n$ or of an abelian variety. This motivated our interest. In \cite{ant}, we advice an isogeny functoriality. This is  the first result of this paper. From the main result of Galateau \cite{Gal}, we deduce:
  \begin{thm}
\label{uno}  Let $V$ be a transverse subvariety of $\ambiente$ of
dimension  $d$. Then,  for any isogeny 
$\phi:\ambiente \to \ambiente$ and any $\eta>0$, 
 $$\mu_{\phi^*\cbo}(V)\ge \cuno \frac{\left({\deg_{\phi^*\cbo}\ambiente}\right)^{\invcd -\eta}}{\left({\deg_{\phi^*\cbo}V}\right)^{\coefmeta}},$$
where $\cuno$ is a positive constant depending on $E^n$ and $\eta$.

\end{thm}
 Since $\mu_{\mathcal{O}_n} (\phi
 V)=\mu_{\phi^*\mathcal{O}_n}(V)$ we can deduce relations of the
 several essential minimi of the image of a variety via isogenies. 
 
  After this first result, we were induced to belive to a more general functorial principle. Let us clarify the setting from another point of view. Only under the geometric assumption that $V$ is transverse, there exist non trivial lower bounds for $\mu(V)$. However, an irreducible variety $V$ is always `relatively'  transverse, meaning that it is transverse in a minimal translate $H$ of an abelian subvariety of the ambient variety $E^N$. If on   $E^N$ we consider the standard polarization, what can one say on the essential minimum of  $V$? In other words,  we consider on $H$ the restriction of the standard polarization of $E^N$. Such a restriction is in general not the standard polarization of $H$.   
 Our main result is an elliptic analogue, up to a remainder term, of a toric conjecture of Amoroso and David.
 \begin{thm}
\label{due} Let  $H$ be the translate of an abelian subvariety of $E^N$ of dimension $n$.
Let $V$ be a $d$-dimensional variety transverse in $H$. Then, for any
$\eta>0$, there exists a positive constant $\ch$ such that 
$$\mu_\cnbundle (V) \ge \ch  \frac{
( \deg_\cnbundle H)^{\invh- \eta}
}{
(\deg_\cnbundle V)^{\invh+ \eta}
}.$$

 \end{thm}
 To prove the theorem, we first assume that $H$ is an abelian variety. This hypothesis is then removed with a simple trick, proposition \ref{fine}.
The main idea of the proof is to approximate $V\subset H$ with ${\bf V}\subset {\bf H}$ in a more convenient position, see section \ref{quattro}. For this we use a `small' transformation, constructed with some technical steps in section \ref{sette}. Then we prove that the restriction of $\mathcal{O}_N$ on ${\bf H}$ is comparable to tensor products of pull backs of the standard polarization on ${\bf H}$, proposition \ref{kercan}. We finally prove that essential minimum and degrees behave well with respect to such operations. For this we use, among other, theorem \ref{uno}.

We would like to  mention that what we actually prove is: 

{\em  if  a quasi-optimal bound for the essential minimum holds for the standard polarization, then a natural analog  holds for the restrictions of the standard polarization to an abelian subvariety. }\\

Finally, in the appendix we present a proposition of Patrice Philippon, which clarifies the relation between functorial conjectures.

In the next section we fix the notation and we prove some basic results. In section \ref{tre} we prove theorem \ref{uno}. In section \ref{quattro} we describe the geometric situation. Thanks to a technical result proven in section \ref{sette}, we approximate our varieties with more convenient ones. For such varieties, we then prove an equivalence of line bundles which represents one of the key ingredients of the proof. In section \ref{sei}  we conclude the proof of theorem \ref{due}. In the appendix we relate our result to other conjectures.\\

{\it Acknowledgments} It is a pleasure for me to thank Patrice Philippon for the appendix and Ga\"el R\'emond for several remarks and contribution.

\section{Preliminaries}

\subsection{Morphisms}
Let $E$ be an elliptic curve defined over the algebraic numbers. We denote by $\rend$ the ring of endomorphism of $E$, that is an order
in a quadratic field. We fix an embedding of $\rend$ in $\co$, and for
$a\in \rend$ we indicate by $|a|$ the standard absolute value in
$\co$. Note that such a value does not depend on the choice of the
embedding. 

Let $n\le N$ be positive integers. Recall that  for a morphism $\psi:\gum^N \to \gum^n$, $$\ker \psi=\{x\in \gum^N\,\,\,:\,\,\,\psi(x)=0\}.$$
To a matrix $\psi \in \mat_{n\times N}(\rend)$ we associate a morphism
$$\psi:\gum^N \to \gum^n,\,\,\, (x_1,\dots, x_N)\to ({\psi_{11}}x_1+\cdots +{\psi_{1N}}x_N,\dots,{\psi_{n1}}x_1+\cdots +{\psi_{nN}}x_N).$$ Viceversa to a morphism $\psi:\gum^N \to \gum^n$, we associate the  matrix defining its kernel
$$\ker \psi=\begin{cases}{\psi_{11}}x_1+\cdots +{\psi_{1N}}x_N=0\\
\vdots\\
{\psi_{n1}}x_1+\cdots +{\psi_{nN}}x_N=0.
\end{cases}$$

\begin{lem}[Kernel relations]
\label{relker}
They hold:
\begin{enumerate}

\item For $\psi\in \mat_{N\times m}(\rend)$ and $\psi' \in \mat_{m'\times N}(\rend)$ with $m,m'\in \mathbb{N}^*$, $$\psi^{-1} \ker \psi'=\ker (\psi'\psi).$$
\item For $a\in \mat_{1\times N-n}(\rend)$, $b\in  \mat_{1\times n}(\rend)$, $$\ker  \begin{pmatrix}\idi_{N-n}&0 \\ a & b\end{pmatrix} =\ker  \begin{pmatrix}\idi_{N-n}&0 \\ 0 & b\end{pmatrix}  .$$
\end{enumerate}
\end{lem}
\begin{proof}
The first relation is  proven by
$$\psi^{-1}\ker \psi'=\psi^{-1}\{x \in \gum^N: \psi'(x)=0\}=\{y \in \gum^N : \psi' \psi (y)=0\}=\ker \psi'\psi.$$ 
The second relation is a simple Gauss reduction.
\end{proof}

We define the norm of a morphism  $\psi:\gum^N\to \gum^n$ as the maximum of the absolute value of the entries of  its associated matrix
$$||\psi||=\max_{ij} |\psi_{ij}|.$$

\subsection{Basic relations of degrees}

In the first instance, we recall basic relations of degrees. Let $n$ and $m$ be positive integers. Let $\elle$ be a symmetric ample line bundle on $E^n$ and let $\elle^m$ be the tensor product of $\elle$, $m$-times. Let  $V$ be an irreducible algebraic subvariety of $\gum^n$ of dimension $d$. Then,
\begin{equation}
\label{gradino1ii}\deg_{\elle^{m}}V=m^{ d}\deg_\elle V.\end{equation}

We now study some relations of degrees under the action of the multiplication morphism, or more in general under the action of an isogeny.
For  $a \in \rend$, we denote by $[a]$ the multiplication by $a$ on $E^n$. Hindry \cite{Hin} Lemma 6 proves
\begin{equation}
\label{hininv}\deg_\elle [a]^{-1}V=|a|^{2(n-d)}\deg_\elle  V
\end{equation}
and
\begin{equation}
\label{hinm} \deg_\elle  [a]V=\frac{|a|^{\abv d}}{|{\rm{Stab}}\,\,\,V\cap \ker[a]|}\deg_\elle V.\end{equation}

Let $\phi:\ambiente \to \ambiente$ be an isogeny. Projection's Formula gives
$$\deg_{\ef^*\elle}V=\deg_\elle \ef_*(V).$$
For a finite set $S$ we denote by $|S|$ its cardinality.  By  \cite{l-b} Corollary (6.6) page 68
 $$\deg_{\phi^*\mathcal{L}}\ambiente=|\ker \phi |\deg_{\mathcal{L}}\ambiente .$$
We now recall a lemma which relates the degree of a variety and of its push-forward under an isogeny of the ambient variety. This relation will be used in several occasions.
 \begin{lem}[\cite{ijnt} Lemma 4.2]
\label{gradino}
 Let $\phi:\ambiente \to \ambiente $ be an isogeny.
 Let $V$ be an irreducible algebraic subvariety of $\ambiente$. Then
 $$\deg_\elle \phi_*(V)=|{\rm{Stab}}\,V\cap \ker \phi|\deg_\elle \phi(V).$$

\end{lem}

\subsection{Basic relations of  essential minimi}

We now investigate useful relations for the essential minimum.
By definition, $h_{\ef^*\elle}(x) = h_\elle(\phi(x))$, then
$$\mu_{\phi^*\mathcal{L}}(V)=\mu_{\mathcal{L}}(\phi_*(V))=\mu_{\mathcal{L}}(\phi(V)).$$
In addition \begin{equation}
\label{gradino1iii}\mu_{\elle^m} (V)=m \mu_\elle(V).\end{equation}
A more interesting property is 

 \begin{lem}
  \label{minessmin} 
  Let $\phi_i$ be isogenies of $E^n$. Then 
 $$\mu_{\otimes_i \phi_i^* \elle}(V) \ge \sum _i \mu_{ \phi_i^* \elle }(V) .$$
  \end{lem}
  
   \begin{proof}
   The proof is worked out by contraddiction. In addition it realizes on the height relation $h_{\otimes_i \phi_i^* \elle} (x) = \sum _i h_{ \phi_i^* \elle} (x) $ for everx $x\in E^n$.
   
   Suppose, by contradiction, that the conclusion of the lemma does not hold. In other words that  
   $ \mu_{\otimes_i \phi_i^* \elle }(V) < \sum _i \mu_{ \phi_i^* \elle} (V) .$  Then there exists a dense subset $U$ of $V$ such that
  \begin{equation}
   \label{stella}  h(U)< \sum _i \mu_{ \phi_i^* \elle} (V) , \end{equation}
   meaning that each element of $U$ has height bounded by $\sum _i \mu_{ \phi_i^* \elle }(V) $. 
    From the definition of $\mu_{ \phi_i^* \elle} (V) $, the set of points of $V$ such that $h_{ \phi_i^* \elle }(x) < \mu_{ \phi_i^* \elle} (V) $ is contained in a closed subset $V_i\subsetneq V$. Since $U$ is dense, $U'=U\setminus{\cup_iV_i\cap U}$ is also dense in $V$. In addition, for every $x\in U'$, $ h_{ \phi_i^* \elle }(x) \ge \mu_{ \phi_i^* \elle}(V) $.  Then, for all $x\in U'$, $h_{\otimes_i \phi_i^* \elle} (x) = \sum _i h_{ \phi_i^* \elle} (x) \ge \sum _i \mu_{ \phi_i^* \elle} (V)$ and $h(U)\ge \sum _i \mu_{ \phi_i^* \elle} (V) $. Which contradicts (\ref{stella}).
  \end{proof}



\section{The isogeny functoriality}
\label{tre}

In this section we concentrate on the dependence of the essential minimum under isogenies. We deduce theorem \ref{uno} from a non-functorial result of Galateau.
\begin{thm}[Galateau \cite{Gal}]
\label{sin1}
Let $\mathcal{O}_n$ be the standard line bundle on $E^n$. For any
$\eta>0$, there exists
a positive constant $\czero$ such that for any
transverse subvariety $V$ of $\ambiente$ of positive  dimension $d$, it holds
$$\mu_{\cbo}(V)\ge \czero (\deg_\cbo V)^{-\invcd-\eta}.$$
\end{thm}
His proof does not extend to other line bundles. In addition he needs to use a choice of a particular basis of global sections. The proof does not work for other choices.

For an isogeny $\phi: \ambiente \to \ambiente$,   we show that  a lower bound like in theorem \ref{sin1} holds for the polarization $\phi*\\mathcal{O}_n$. An basic role in the proof is played by  stabilizers.
An isogeny is a group morphism,
 thus the stabilizer of the image or preimage of a variety can be
 easily understood. On one hand, the stabilizer of $V$ is related to 
 the fiber of $\phi$ on $V$. On the other hand, stabilizers are related to the
 degree of the variety via the well-known formulas  (\ref{hinm}).
 We define a special variety depending on $V$ and $\phi$. Such a
 variety  allows us to produce a kind of reverse degree formula. This
 will be a key ingredient to prove theorem \ref{uno}.
Let us prove a first elementary lemma.
\begin{lem}
\label{ordinestab}
Let $\phi:\ambiente \to \ambiente $ be an isogeny. Then,
\begin{enumerate}
\item $\stab \phi^{-1}(V)=\phi^{-1}(\stabv)$.

\item Let $\hat\phi$ be an isogeny such that  $\hat\phi \phi=\phi \hat\phi=[a]$. Then  $$|\stab \hat\phi^{-1} (V) \cap \ker[a]|=|\ker\hat\phi| |\stabv \cap \ker \phi|.$$
\end{enumerate}
\end{lem}
\begin{proof}
i. Let $t\in \stab \phi^{-1}(V)$ then $$ \phi^{-1}(V)+t \subset  \phi^{-1}(V).$$ Taking the image, $V+\phi(t) \subset V$ and $\phi(t) \in \stab V$. That gives $t \in \phi^{-1}(\stabv)$. Conversely, let $t \in \stabv$, then $$V+t\subset V$$ and taking the preimage $\phi^{-1}(V+t) \subset \phi^{-1}V$. Then $\phi^{-1}(t) \subset \stab \phi^{-1}(V)$.

ii. By part i. applied to $\hat\phi$, we have $\stab \hat\phi^{-1} (V)=\hat\phi^{-1} (\stabv)$. As $\hat\phi \phi=[a]$, $\ker[a]=\hat\phi^{-1}(\ker\phi)$.
Then
$$\stab \hat\phi^{-1} (V) \cap \ker[a]=\hat\phi^{-1}\left(\stabv \cap \ker \phi \right).$$

\end{proof}

We are now ready to prove the isogeny functoriality.  For the convenience of the reader we recall the statement.
 \begin{thm1}
Let $V$ be a transverse subvariety of $\ambiente$ of
dimension  $d$. Then,  for any isogeny 
$\phi:\ambiente \to \ambiente$ and any $\eta>0$, 
 $$\mu_{\phi^*\cbo}(V)\ge \cuno \frac{\left({\deg_{\phi^*\cbo}\ambiente}\right)^{\invcd -\eta}}{\left({\deg_{\phi^*\cbo}V}\right)^{\coefmeta}},$$
where $\cuno=3^{-\frac{n}{n-d}}c_0(E^n, 2(n-d)\eta)$ and $\czero$ is as in theorem \ref{sin1}.

\end{thm1}

\begin{proof}
Let  $a$ be an integer of minimal absolute value such that there exists an isogeny  $\efdual$  with $\ef \efdual =\efdual \ef=[a]$. Then by definition of dual isogeny $|a|\le|\det \phi|$.
Define $$W={\efdual}^{-1}(V).$$
We have 
$$[a] W=\ef \efdual W= \ef \efdual{\efdual}^{-1}(V) =\ef(V).$$
Then \begin{equation}
\label{miniesi}\mu_{\ef^*\cbo}(V)=\mu_\cbo(\phi(V))=\mu_\cbo([a]W)=|a|^\abv  \mu_\cbo(W).
\end{equation}

In order to apply  theorem \ref{sin1}, we estimate $\deg_\cbo W$.
By formula (\ref{hinm}),
$$\deg_\cbo \ef (V)=\deg_\cbo [a]W=\frac{|a|^{\abv d}}{|\stab W \cap \ker[a]|}\deg_\cbo W$$
or equivalently
$$\deg_\cbo W=\frac{|\stab W \cap \ker[a]|}{|a|^{\abv d}}\deg_\cbo \ef (V).$$
 Using Lemma \ref{ordinestab}  ii. and Lemma \ref{gradino}, we obtain
\begin{equation*}
\begin{split}
\deg_\cbo W &=\frac{|\ker \efdual|}{|a|^{2d}}  |\stab V \cap \ker \phi| \deg_\cbo \ef (V)\\  
&=\frac {|\ker \efdual| }{|a|^{2d}} \deg_\cbo \ef_* (V).
\end{split}
\end{equation*}
We now estimate $\mu_\cbo(W)$ using theorem \ref{sin1}.  For simplicity we denote $\czerosp=\czero$. Note that
isogenies preserve dimensions and trasversality, so $\dim W=\dim V=d$
and $W$ is transverse. We obtain
\begin{equation*}
\begin{split}
\mu_\cbo(W)& \ge {\czerosp}\left({\deg_\cbo W}\right)^{-\invcd+\eta}\\
&= {\czerosp}\left(\frac{{|a|^{\abv d}} } {|\ker \efdual| \deg_\cbo \ef_* (V)  } \right)^{{\coefmeta}}\\
&= {\czerosp}\left(\frac{{|a|^{\abv d}} \deg_\cbo \ambiente} {|\ker \efdual| \deg_\cbo \ef_* (V)  } \right)^{{\coefmeta}}(\deg_\cbo \ambiente)^{\frac{-1}{n-d}-\eta}\\
&=3^{\frac{-n}{n-d}}{\czerosp}\left(\frac{{|a|^{\abv d}} \deg_\cbo \ambiente} {|\ker \efdual| \deg_\cbo \ef_* (V)  } \right)^{{\coefmeta}}(\deg_\cbo \ambiente)^{-\eta},\\
\end{split}
\end{equation*}
indeed $\deg_\cbo \ambiente=3^n$
We  substitute this last estimate in (\ref{miniesi}), then
\begin{equation*}
\begin{split} \mu_{\ef^*\cbo}(V)&=|a|^\abv  \mu_\cbo(W)\\&\ge |a|^\abv 3^{\frac{-n}{n-d}}{\czerosp}\left(\frac{{|a|^{\abv d}} \deg_\cbo \ambiente} {|\ker \efdual| \deg_\cbo \ef_* (V)  } \right)^{{\coefmeta}}(\deg_\cbo \ambiente)^{-\eta}\\
&=3^{\frac{-n}{n-d}}{\czerosp}\left(\frac{|a|^{\abv n-\abv d} |a|^{\abv d}|\ker \phi| \deg_\cbo \ambiente} {|\ker \efdual| |\ker \phi|\deg_\cbo \ef_* (V)  } \right)^{\coefmeta} (\deg_\cbo \ambiente)^{-\eta}|a|^{-\abv  (\codv) \eta}\\
&=3^{\frac{-n}{n-d}}{\czerosp}\left(\frac{|a|^{\abv n}\deg_{\ef^*\cbo}\ambiente} {|a|^{\abv n}\deg_\cbo \ef_* (V)  } \right)^{\coefmeta} (\deg_\cbo \ambiente)^{-\eta}|a|^{-\abv  (\codv) \eta}
\\&=3^{\frac{-n}{n-d}}{\czerosp}\left(\frac{\deg_{\ef^*\cbo}\ambiente} {\deg_{\ef^*\cbo}  V  } \right)^{\coefmeta}(\deg_\cbo \ambiente)^{-\eta}|a|^{-\abv  (\codv) \eta},
\end{split}
\end{equation*}
where  $|\ker \efdual| |\ker \phi|=|a|^{2n}$ because $\phi\hat\phi=\hat\phi\phi=[a]$. In addition $$\deg_{\ef^*\cbo}\ambiente=|\ker \phi|\deg_{\cbo}\ambiente=|\det
\phi|  \deg_\cbo \ambiente .$$
Since $|a| \le |\deg \phi|$, $$(\deg_\cbo \ambiente)^{-\eta}|a|^{-\abv  (\codv) \eta}\ge
(\deg_\cbo \ambiente|a|)^{-\abv  (\codv)
  \eta}\ge(\deg_{\ef^*\cbo}\ambiente)^{-\abv  (\codv) \eta}.$$
Then
$$\mu_{\ef^*\cbo}(V)\ge3^{\frac{-n}{n-d}}{\czerosp}\frac{\left(\deg_{\ef^*\cbo}\ambiente\right)^{\frac{1}{n-d}-\abv  (\codv) \eta}} {\left(\deg_{\ef^*\cbo}  (V) \right)^{\coefmeta } }.
$$
 This  easily implies the wished bound.
\end{proof}

Note that,  theorem \ref{uno}  implies  theorem \ref{sin1},
 simply choosing $\phi=id$.  So, we have proven 
 \begin{cor}
 Theorem \ref{sin1} and theorem \ref{uno} are equivalent.
\end{cor}

\section{The minimal abelian subvariety containing $V$}
\label{quattro}
In this section we assume that $H$ is
an abelian subvariety of $E^N$ of dimension $n$. In proposition \ref{fine} we will see how to adapt our argument to the general case of a translate of an abelian subvariety. 

\begin{lem}\label{mw}

Let $H$ be an abelian subvariety of $E^N$ of dimension $n$. Then there exists an isogeny
 \begin{equation*}
\varphi={\varphi_H\choose\varphi_{H'}}: \gum^N\to \gum^N.\end{equation*}
such that:
\begin{enumerate}
\item The rank of $\varphi_H$ is $N-n$ and the rank of  $\varphi_{H'}$ is $n$, 
\item $\varphi(H)=\{0\}^{N-n}\times E^n$.
\item $|\det \varphi|\le \bertrand , $ for $ \bertrand \le 3^{2N} 2^{8nN^2}$.
\end{enumerate}
\end{lem}

\begin{proof}
By a lemma of Betrand, (see the appendix of \cite{bertrand} proposition 5.1). We can find a complement $H'$ of $H$ such that $H+H'=E^N$ and  the cardinality  of $H\cap H' $ is bounded by a constant $ \bertrand_1$ depending only on $n,N$. A closer inspection of the proof of this lemma, shows that eventualy $ \bertrand_1 \le 2^{8nN^2}$.

By Masser and W\"ustholz \cite{Masserwustholz} Lemma 1.3, there exists  a matrix  $\varphi_H\in \mat_{{(N-n)}\times
  N}(\rend)$   of rank $N-n$ such that  $\ker \varphi_H= H+\tau$ for $\tau$ a torsion group contained in $H'$ of cardinality bounded by $\deg_{\mathcal{O}_N}E^N=3^N$. 
Similarly, let  $\varphi_{H'}\in \mat_{{n}\times
  N}(\rend)$   be a matrix of rank $n$ such that $\ker \varphi_{H'}= {H'}+\tau'$ for $\tau'$ a torsion group contained in $H$ of cardinality bounded by  $3^N$ . Define the isogeny 
 $$\varphi={\varphi_H\choose\varphi_{H'}}: \gum^N\to \gum^N.$$
Then 
$$\ker \varphi= (H+\tau) \cap( H'+\tau') =(H\cap H') + \tau+\tau'\le \bertrand_13^{2N}.$$ 
\end{proof}

Using the isomorphism produced in the
appendix, we are going  to approximate $H$ with a $\bf{H}$ in a convenient position
in $E^N$. 
\begin{lem}
\label{tras} Let $\varphi$ be the isogeny defined in lemma \ref{mw}. Then,
there exists an isomorphism $T: \gum^N\to \gum^N$ such that:
\begin{enumerate}
\item$||T||,||T^{-1}||\le\frac{1}{N}{N\choose n}$
\item All $n\times n$ minors of the matrix consisting of the last $n$ columns of $\left(\varphi T\right)^{-1}$ are different from zero.
\end{enumerate}
\end{lem}
\begin{proof}
Apply proposition \ref{cool} to the matrix $\psi$ consisting of the last $n$ rows of $(\varphi^{-1})^t$, the transpose of the inverse of $\varphi$. Then there exists a permutation matrix $J$ and  a matrix $T_0$ such that:
\begin{itemize}
\item[-] $T_0=\begin{pmatrix}\idi_n&X\\ 0&\idi_{N-n}\end{pmatrix} $ and $|X_{ij}| \le \frac{1}{N}{N\choose n}$, 
\item[-] All the $n\times n$ minors of $(\varphi^{-1})^tJT_0$ are different from zero.
\end{itemize}
Then all $n\times n$ minors of the last $n$ columns of $T_0^tJ^t\varphi^{-1}$ are non zero.
Note  that $T_0^tJ^t\varphi^{-1}=(\varphi (J^{-1})^t(T_0^{-1})^t)^{-1}$.  In addition $T_0^{-1}=\begin{pmatrix}\idi_n&-X\\ 0&\idi_{N-n}\end{pmatrix} $. 
Thus the lemma is proven for $T=(J^{-1})^t(T_0^{-1})^t$.
\end{proof}

This isomorphism is 
particularly nice. For our problem of estimating the essential
minimum, we will see that it  is equivalent to work in the domain or
in the codomain of $T$,   up to a constant depending on $N$ and $n$.  This is a consequence of the fact that the entries of $T$ are bounded by a constant and of  an estimate by Masser and W\"ustholz. This estimate relates the degree of a variety and of its  push-forward via an isogeny.
\begin{lem}[\cite{MW} Lemma 2.3]
\label{gradino1}
 Let $\psi:\gum^N\to \gum^N$ be an isogeny. Let  $X$ be an irreducible algebraic subvariety of $\gum^N$ of dimension $d$. Then, 
$$\deg_\elle \psi_*(X) \le N^d(3N ||\psi||)^{2 d}\deg_\elle X.$$
\end{lem}

We can then easily deduce the following:

\begin{lem}
\label{stima}
Let $T$ be as in lemma \ref{tras}. Then
\begin{enumerate}
\item $(N||T||)^{-\abv}{h_\elle(x)}\le h_\elle(T^{-1}x)\le (N ||T^{-1}||)^{\abv}h_\elle(x)$, for every $x\in E^N$,
\item $\deg_\elle T^{-1}(H)\ge (9N^3 ||T||^2)^{-n}\deg_\elle H$, 
\item $\deg_\elle T^{-1}V\le (9N^3||T^{-1}||^2)^{d}\deg_\elle V$.
\end{enumerate}
\end{lem}
\begin{proof}
The second inequality of part i. is simply given by the triangle inequality. The triangle inequality also gives
$h_\elle(x)=h_\elle(TT^{-1}(x))\le(N ||T||)^{\abv}h_\elle(T^{-1}(x))$. This is the first inequality of i.

To prove part ii. apply lemma \ref{gradino1}   with $X=T^{-1}H$, then
\begin{equation*}
\begin{split}\deg_\elle H=\deg_\elle TT^{-1}(H)&\le N^n(3N ||T||)^{\abv n} \deg_\elle T^{-1}(H)\\&= N(9N^3 ||T||^2)^{ n} \deg_\elle T^{-1}(H).
\end{split}
\end{equation*} Note that $T^{-1}$ is an isomorphism so $T^{-1}H$ is irreducible.

Part iii. is an immediate application of Lemma \ref{gradino1}  with $X=V$.
\end{proof}

Thanks to the isomorphism $T$, we can give to the abelian subvariety $H$ a convenient position, in the sense that follows.
We construct a matrix $\phi$, which has the property that all minors of the last $n$ columns are non zero. Moreover its entries are close to the entries of $\varphi$.
\begin{D} Let $\varphi$ be as in lemma \ref{mw} and let $T$  be as in lemma \ref{tras}.
We define the abelian subvariety $$\mathbf{H}=T^{-1} H$$ and the isogeny $$\phi=\varphi T={\phi_H\choose\phi'_H},$$ 
where $\phi_H\in \mat_{N-n\times N}(\rend)$.

\end{D}
The kernel relation, immediately gives \begin{equation*}\label{cic}\mathbf{H}=T^{-1} H\subset \ker \phi_H.\end{equation*}
Note that, the isogeny $\phi:\gum^N \to \gum^N$ sends $H$ to the last $n$ factors, $$\phi(\mathbf{H})=0\times \dots \times 0 \times \gum^n.$$  Indeed $\phi(\mathbf{H})=\varphi(H)$, which has  by construction such a property, see lemma \ref{mw} ii.
We denote the immersion on the last $n$ factors by  $$i: \gum^n \to \gum^N, \,\,\,(x_1,\dots,x_n)\to (0,\dots,0,x_1,\dots,x_n).$$
 An immediate consequence of lemmas \ref{gradino} and \ref{mw}. is \begin{cor}
\label{push}
Let $\elle$ be a symmetric ample line bundle on $E^n$. Then, for $\phi$ as above,
$$\deg_\elle i^* \phi_*({\bf{H}})=|\det \phi|\deg_\elle E^n\le \bertrand \deg_\elle E^n.$$
\end{cor}
\begin{proof}
We first remark that $|\det \phi|\le \bertrand$. This simply follows by lemma \ref{mw} iii. and the fact that $T$ is an isomorphism. Now, 
apply lemma  \ref{gradino} to  $\phi$ and ${\bf{H}}$. Then $\deg_{i_*\elle}  \phi_*({\bf{H}})=|\det \phi|\deg_{i_
*\elle} \phi({\bf{H}})$.  Note that the map $i$ preserves the degree of subvarieties of $\{0\}^{N-n}\times E^n$.  In addition $\phi({\bf{H}})=E^n$. So $|\det \phi|\deg_{i
_*\elle} \phi({\bf{H}})=|\det \phi|\deg_\elle E^n$. \end{proof}

\begin{D}
\label{defa}
Let $\defalf$ be the minimal  positive integer  such that
there exists an isogeny $\hat\phi$ satisfing
$\phi\hat\phi=\hat\phi\phi=[\defalf]$.
We decompose $\invf=(A|B)$ with $A\in \mat_{N \times (N-n)}(\rend)$ and $B\in \mat_{N\times n}(\rend)$. We denote by $a_i$ the $i$-th row of $A$, similarly
 $$B=\left(\begin{array}{c}b_1\\
b_2\\
\vdots\\
b_N\end{array}\right).$$
\end{D}
 Note that, by definition of dual
isogeny $\defalf\le |\det \phi|$.
\begin{lem}
\label{isogenie}
For $I \in \mathbb{I}=\{(i_1,\dots, i_n)\,\,\,:\,\,\,i_j\in \{1,\dots , N\} \,\,\,{\rm{and}}\,\,\,i_j<i_{j+1}\}$ the morphism 
 $$\phi_I=\left(\begin{array}{c}b_{i_1}\\
\vdots\\
b_{i_n}\end{array}\right):\gum^n\to \gum^n$$ is an isogeny. 
\end{lem}
\begin{proof}
 As $\phi \invf=\defalf \idi_N$,  lemma \ref{tras} ii. implies that all $n\times n$ minors  of $B$ are non-zero.  Then $\det \phi_I\not=0$. This is equivalent to say that $\phi_I$ is an isogeny. 
\end{proof}

\section{An equivalence  of line bundles}
\label{cinque}
 In this section, we work with the canonical line bundle $\mathcal{O}_N$ on the ambient variety $E^N$. For $H$ an abelian subvariety,  we study
$\mathcal{O}_{N| H}$. If $H$ has a sufficiently general position, we can express $\mathcal{O}_{N| H}$ as tensor
products and bull-backs via isogenies of the canonical bundle
$\mathcal{O}_n$ on $H$. 

Let $\elle$ and $\mathcal{M}$ be line bundles. We denote by $\chern(\elle)$ a representative of the first Chern-class of $\elle$. By $\piuc$ we mean the sum of cycles and by $\elle^m$ we mean the tensor product of $\elle$ $m$-times.
Recall that $\chern(\elle\otimes \mathcal{M})=\chern(\elle)\piuc \chern(\mathcal{M}).$ For $f$ a morphism
$
\chern(f^*\elle)=f^{-1}\chern(\elle).$

We recall that  $\cani$ is the line bundle on $E$ defined by the neutral element. 
 The standard line bundle $\cbo$ on $\gum^n$  is the tensor product of the pull-back of $\cani$ via the natural projections. 
Let $e_i:E^n \to E$ and $f_i:E^N \to E$ be the  projections on the $i$-th factor. Note that $e_i$ and $f_i$ are the vectors of a standard basis of $\ze^n $ and $\ze^N$.
By definition of standard line bundle, 
\begin{enumerate}
\item[ ]$\chern(\cbo)=\piuc_{i=1}^n  \ker e_i$,
\item[ ] $\chern(\cnbundle)=\piuc_{i=1}^N  \ker f_i$.
\end{enumerate}
 Note that, for any integer $\defalf$,  $\ker (\defalf f_i) =\ker (f_i[\defalf])=[\defalf]^{-1}\ker f_i=[\defalf]^*\ker f_i$. In addition, by \cite{l-b} page 34 corollary 3.6, for $\elle$ a symmetric ample line bundle on $E^n$,  
\begin{equation*}
\chern([\defalf]^*\elle)=\chern(\elle^{\defalf^2}).
\end{equation*}
Then, 
\begin{enumerate}
\item[ ] $\chern(\mathcal{O}_N^{\defalf^2})=\piuc_{i=1}^N  \ker \defalf f_i$.
\end{enumerate}

Let us first state a basic relation which proves useful. 

\begin{lem}
\label{kercan}
In the above notation, it holds
$$\chern({\mathcal{O}_N^{\defalf^2}}_{|\mathbf{H}})=\piuc_{i=1}^N\ker{\phi_H \choose\defalf f_i}.$$
\end{lem}
\begin{proof}
Note that,
\begin{equation*}
 \begin{pmatrix} \idi_{N-n}&0\\ a_i & b_i\end{pmatrix} \phi={\phi_H \choose \defalf f_i},\end{equation*} 
 where $\defalf$ is given in definition \ref{defa}.
 From lemma \ref{isogenie}, $b_i\not=0$ for all $1\le i\le N$. Thus the rank of $ \begin{pmatrix} \idi_{N-n}&0\\ a_i & b_i\end{pmatrix} $ is $N-n+1$. Since $\phi$ is invertible also ${\phi_H \choose  \defalf f_i}$ has rank $N-n+1$. This means that $\ker \defalf f_i$  intersects generically $\mathbf{H}$. Then, the restriction bundle ${\mathcal{O}_{N}^{\defalf^2}}_{|\mathbf{H}}$ satisfies   $\chern({\mathcal{O}_{N}^{\defalf^2}}_{|\mathbf{H}})= \mathbf{H} \cap \chern({\mathcal{O}_{N}^{\defalf^2}})=\mathbf{H} \cap \left(\piuc\ker {\defalf} f_i\right)=\piuc_{i=1}^N\ker{\phi_H \choose {\defalf}f_i}.$

\end{proof}

We sum up the situation with the following diagram:
\begin{equation*}
\begin{array}{ccccccc}
E^N &\stackrel{\phi}{\longrightarrow}  & E^{N-n}\times E^n & \stackrel{i}\hookleftarrow & E^n & \stackrel{\phi_I}{\longrightarrow}   E^n\\
\mathbf{H}& \longrightarrow & 0\times E^n & \longleftarrow & E^n & &
\end{array}
\end{equation*}

\begin{thm}
\label{relchiave}
The following equivalence of line bundles holds
$$ {\mathcal{O}_{{N}_{\,\,\,{\big{|\mathbf{H}}}}}^{ \defalf^2 \moltp }
}\cong \phi^*i_* \bigotimes_I \phi^*_I \cbo,$$
where $\phi\hat\phi=\hat\phi \phi=[\defalf]$ is as in definition \ref{defa} and $\phi_I$ is defined in
lemma \ref{isogenie}. 
\end{thm}
\begin{proof}
Two line bundles are equivalent if and only if they have the same Chern-class. Recall that $$\chern(\cbo)=\piuc_{i=1}^n \ker e_i.$$
For any  isogeny $\psi:\gum^n \to \gum^n$, $$\chern(\psi^*\cbo)=\psi^{-1} \chern(\cbo)=\piuc_{i=1}^n \ker (e_i \psi)=\piuc_{i=1}^n \ker (\psi_i),$$ where 
$\psi_i: \gum^n \to \gum$ is  the $i$-th row of $\psi$.

 Apply this formula to each $\phi_I$. Then,  for $I=(i_1,\dots ,i_n)$,
 \begin{equation}\label{cherni}\chern(\phi_I^*\cbo )= \ker b_{i_1} \piuc\dots \piuc\ker b_{i_n}.\end{equation} 
 Since the Chern class of the tensor product is the sum of the Chern classes, we obtain
   \begin{equation}\label{chernl}\chern\left(\bigotimes_I \phi^*_I\cbo \right)= \piuc_{I\in \mathbb{I}} (\ker b_{i_1} \piuc\dots \piuc\ker b_{i_n}) =\frac{n}{N}{N \choose n} \piuc_{i=1}^N \ker b_i,\end{equation}
  where the last equality  is justified from the fact that the cardinality of $\mathbb{I}$ is ${N \choose n}$, in addition each multi-index $I$ consists of  $n$ coordinates and  each of the $N$ indices  appears with the same recurrence, so each row $b_{i_j}$ appears $\frac{n}{N}{N \choose n} $-times.
  Note that, for a bundle on $E^n$, $\chern(i_*\elle)=\{0\}^{N-n}\times \chern(\elle)$. 
In view of Lemma \ref{relker} ii., we deduce  
\begin{equation*}
\begin{split}\chern(i_*\bigotimes_I \phi^*_I\cbo)&= \moltp \piuc _{i=1}^N \ker \begin{pmatrix} \idi_{N-n}&0 \\ 0 & b_i\end{pmatrix}\\&= \moltp \piuc _{i=1}^N \ker \begin{pmatrix}\idi_{N-n}&0 \\ a_i&b_i\end{pmatrix}.
\end{split}
\end{equation*}
Recall that $\hat\phi\phi=[\defalf]$. Using  lemma \ref{relker} i. and  lemma \ref{kercan}, we conclude
\begin{equation*}
\begin{split}\chern(\phi^*i_* \bigotimes_I \phi^*_I\cbo)&=\phi^{-1}\chern(i_* \bigotimes_I \phi^*_I\cbo)\\ &
= \moltp \piuc _{i=1}^N\ker \begin{pmatrix} \idi_{N-n}&0 \\ a_i & b_i\end{pmatrix}  \phi\\&=
 \moltp \piuc _{i=1}^N\ker {\phi_H \choose \defalf f_i}\\&
=\moltp \chern\left({\mathcal{O}_N^{ \defalf^2}}_{|H}\right)=\chern\left({\mathcal{O}_N^{ \defalf^2\moltp}}_{|H}\right)
\end{split}
\end{equation*}

 \end{proof}

I am gratefull to Ga\"el R\'emond for suggesting me the following proof.
\begin{propo}
\label{gael} The following equivalence of degrees holds
$$\deg  \bigotimes_{I \in \mathbb{I}} \phi^*_I \cbo = \moltp^n \sum_{I \in \mathbb{I}} \deg  \phi_I^*\cbo.$$
\end{propo} 
\begin{proof}

 We  compute the degrees as intersection numbers. By relation
 (\ref{cherni}), we have
$$\deg \phi_I^* \cbo=n! \prod_{i_j\in I}\ker b_{i_j},$$
where the product has the sens of intersection number.
Similarly, by formula (\ref{chernl}), we obtain
$$ \deg  \otimes_I \phi_I^* \cbo=n! \moltp^n
\sum_{i_1<\dots <i_n}\prod_{j=1}^n \ker b_{i_j}=
\moltp^n\sum_I\deg \phi_I^*\cbo.$$

\end{proof}

\section{The proof of theorem \ref{due}: The conclusion}
\label{sei}
We first prove a weak form of theorem \ref{due}. We then remove the restrictive hypothesis.
\begin{thm}
\label{canoni} Theorem \ref{due} holds for $H$ an abelian subvariety and an explicit positive constant
$c'(E,n,N,\eta)$.
\end{thm}

\begin{proof}

First we prove the theorem for 
$${\mathbf{H}}=T^{-1}H=\ker \phi_H$$ and
$${\mathbf{V}}=T^{-1}V,$$
where $T$ is as in lemma \ref{tras}.
The isomorphism $T$ preserves transversality and dimensions. So  $\dim{\mathbf{H}}=n $, $\dim{\mathbf{V}}=d $ and $N$ is the dimension of the ambient variety.
Since  ${\mathbf{V}}$ is transverse in ${\mathbf{H}}$, $i^{*}\phi_*({\mathbf{V}})$ is transverse in $\gum^n$.

By lemma \ref{minessmin}, we know that 
$$\mu_{\bigotimes_I\phi^*_I \cbo}(i^{*}\phi_*({\mathbf{V}})) \ge \sum_I
\mu_{\phi^*_I \cbo}(i^{*}\phi_*({\mathbf{V}})).$$

We apply Theorem \ref{uno}  to each $\Phi_I$ on $E^n$ and $\cbo$. For simplicity we denote $\cunosp=\cuno$. We deduce that 
for each $I$, 
\begin{equation*}
\mu_{\phi^*_I \cbo}(i^{*}\phi_*({\mathbf{V}})) \ge   
\cunosp \frac{\left(\deg_ { \phi^*_I \cbo}E^n\right)^{\frac{1}{n-d}-\eta}
}{\left(\deg_ {\phi^*_I
      \cbo}i^{*}\phi_*({\mathbf{V}})\right)^{\frac{1}{n-d}+\eta}}.\end{equation*}
We obtain
\begin{equation*}
\begin{split}
\mu_{\bigotimes_I\phi^*_I \cbo}(i^{*}\phi_*({\mathbf{V}})) &=\sum_I
\mu_{\phi^*_I \cbo}(i^{*}\phi_*({\mathbf{V}}))\\
&\ge   
 \cunosp\sum_I \frac{\left(\deg_ { \phi^*_I \cbo}E^n\right)^{\frac{1}{n-d}-\eta}
}{\left(\deg_ {\phi^*_I \cbo}i^{*}\phi_*({\mathbf{V}})\right)^{\frac{1}{n-d}+\eta}}
\end{split}
\end{equation*}
Since each bundle is ample, for every variety $X$, we have   $\deg_ {\phi^*_I
      \cbo  }X \le \deg_ {\otimes_I\phi^*_I
      \cbo}X $. Recall that, for $x_i\ge 1$, $(\sum x_i)^{1/m}\le \sum
    x_i^{1/m}$. Using then Proposition \ref{gael}, we deduce
\begin{equation*}
\begin{split}
\sum_I \frac{\left(\deg_ { \phi^*_I \cbo}E^n\right)^{\frac{1}{n-d}-\eta}
}{\left(\deg_ {\phi^*_I \cbo}i^{*}\phi_*({\mathbf{V}})\right)^{\frac{1}{n-d}+\eta}}&\ge  \frac{\left(\sum_I \deg_ { \phi^*_I \cbo}E^n\right)^{\frac{1}{n-d}-\eta}
}{\left( \deg_ {\otimes_I\phi^*_I
      \cbo}i^{*}\phi_*({\mathbf{V}})\right)^{\frac{1}{n-d}+\eta}}\\&\ge \moltp^{\frac{-n}{n-d}+n\eta} \frac{\left(\deg_ { \bigotimes_I\phi^*_I \cbo}E^n\right)^{\frac{1}{n-d}-\eta}
}{\left(\deg_ {\bigotimes_I\phi^*_I \cbo}i^{*}\phi_*({\mathbf{V}})\right)^{\frac{1}{n-d}+\eta}}.
\end{split}
\end{equation*}
Therefore
\begin{equation*}\mu_{\bigotimes_I\phi^*_I \cbo}(i^{*}\phi_*({\mathbf{V}})) \ge 
\cunosp\moltp^{\frac{-n}{n-d}+n\eta} \frac{\left(\deg_ { \bigotimes_I\phi^*_I \cbo}E^n\right)^{\frac{1}{n-d}-\eta}
}{\left(\deg_ {\bigotimes_I\phi^*_I \cbo}i^{*}\phi_*({\mathbf{V}})\right)^{\frac{1}{n-d}+\eta}}
.
\end{equation*}

Note that $E^n=i^{*}\phi({\mathbf{H}})$. Moreover, by corollary \ref{push}, $\deg E^n \le \bertrand^{-1} \deg i^{*}\phi_*({\mathbf{H}})$. Define $$\cunospdue=\cunosp \bertrand^{-1}.$$  We deduce  \begin{equation*}\mu_{\bigotimes_I\phi^*_I \cbo}(i^{*}\phi_*({\mathbf{V}})) \ge 
\cunospdue\moltp^{\frac{-n}{n-d}+n\eta} \frac{\left(\deg_ { \bigotimes_I\phi^*_I \cbo}i^{*}\phi_*({\mathbf{H}})\right)^{\frac{1}{n-d}-\eta}
}{\left(\deg_ {\bigotimes_I\phi^*_I \cbo}i^{*}\phi_*({\mathbf{V}})\right)^{\frac{1}{n-d}+\eta}}
.
\end{equation*}

By   theorem \ref{relchiave} and relations (\ref{gradino1ii}) and
(\ref{gradino1iii}) we obtain
\begin{equation*}
\begin{split}
 \deg_ {\bigotimes_I \phi^*_I \cbo}i^{*}\phi_*({\mathbf{H}}) &=\deg_{\cnbundle_{| {\mathbf{H}}}^{ \defalf^2 \moltp } }{\mathbf{H}}=\left[ \defalf^2\moltp \right]^n\deg_ {\cnbundle}{\mathbf{H}},\\
\deg_ {\bigotimes_I \phi^*_I \cbo}i^{*}\phi_*({\mathbf{V}})&=\deg_{\cnbundle_{| {\mathbf{H}}}^{ \defalf^2 \moltp } }{\mathbf{V}}=\left[\defalf^2\moltp\right]^d\deg_ {\cnbundle}{\mathbf{V}},\\
\mu_{ \bigotimes_I \phi^*_I \cbo}(i^{*}\phi_*({\mathbf{V}}))&=\mu_{\cnbundle_{| {\mathbf{H}}}^{ \defalf^2 \moltp } }({\mathbf{V}})=\left[\defalf^2\moltp\right]\mu_\cnbundle ({\mathbf{V}}).
\end{split}
\end{equation*}
Then
\begin{equation*}
\defalf^2\moltp\mu_\cnbundle ({\mathbf{V}})\ge \cunospdue \left[\defalf^2\moltp\right]^{1-(n+d)\eta} \moltp^{\frac{-n}{n-d}+n\eta} 
\frac{
\left(\deg_ {\cnbundle}{\mathbf{H}}\right)^{\invh-\eta}}{\left(\deg_ {\cnbundle}{\mathbf{V}}\right)^{\invh+\eta}}.
\end{equation*}
In conclusion
$$\mu_\cnbundle ({\mathbf{V}})\ge \cunospdue \defalf^{-2(n+d)\eta} \moltp^{\frac{-n}{n-d}-d\eta}
\frac{(\deg_ {\cnbundle}{\mathbf{H}})^{\invh-\eta}}{(\deg_ {\cnbundle}{\mathbf{V}})^{\invh+\eta}}.$$
By Lemma \ref{mw},  $\defalf \le |\det \phi|\le \bertrand$.  Then, for $c_3=c_2 \bertrand^{2(n-d)\eta} \moltp^{\frac{-n}{n-d}-d\eta}$,
$$\mu_\cnbundle ({\mathbf{V}})\ge c_3
\frac{(\deg_ {\cnbundle}{\mathbf{H}})^{\invh-\eta}}{(\deg_ {\cnbundle}{\mathbf{V}})^{\invh+\eta}}.$$

So we just proved  the theorem for $\mathbf{V}$ and $\mathbf{H}$. We are now going to show it for $V$ and $H$. This is a consequence of the fact that the isomorphism $T$ does not change much degrees and heights (see lemma \ref{stima}).

By lemma \ref{stima} i. we deduce that $\mu_\cnbundle(V)\ge \frac{\mu_\cnbundle({\mathbf{V}})}{({N||T||})^{\abv} }$, indeed isomorphisms preserve non-dense subsets of varieties. Then
$$\mu_\cnbundle (V)\ge c_3 (N||T||)^{-\abv} \frac{(\deg_ {\cnbundle}{\mathbf{H}})^{\invh-4(n+d+1)\eta}}{(\deg_ {\cnbundle}{\mathbf{V}})^{\invh+\eta}}.$$
By Lemma \ref{stima} ii. and iii. we  obtain
\begin{equation*}
\begin{split}\mu_\cnbundle (V)\ge  & c_3 (N||T||)^{-2}\\&{(9N^3||T||^2)^{-\frac{n}{n-d}+n\eta}(9N^3||T^{-1}||^2)^{-\frac{d}{n-d}-d\eta}} \frac{(\deg_ {\cnbundle}H)^{\invh-\eta}}{(\deg_ {\cnbundle}V)^{\invh+\eta}}.\end{split}
\end{equation*}
 In view of lemma \ref{tras} i., $||T||,||T^{-1}||\le \frac{1}{N} {N\choose n}$. Thus
$$\mu_\cnbundle (V)\ge  c_4  \frac{(\deg_ {\cnbundle}H)^{\invh-\eta}}{(\deg_ {\cnbundle}V)^{\invh+\eta}},$$
where $$c_4=c_3 (9N)^{\frac{-n-d}{n-d}+(n-d)\eta}{N\choose n}^{\frac{-4n}{n-d}-2(n-d)\eta}.$$
This directly gives the wished result.
\end{proof}

We now relax the hypothesis on $H$. We prove theorem \ref{due} for $V$ transverse in a translate of an abelian subvariety.
\begin{propo}
\label{fine}
Theorem \ref{due} holds for $H$ a translate of an abelian subvariety  and  a positive constant
$c''= \frac{1}{2}c',$ where $c'$ is as in theorem \ref{canoni}.
\end{propo}

\begin{proof}
Suppose that $V$ is transverse in a translate $H$ of an abelian subvariety. Define $$\theta=c'   \frac{
( \deg_\cnbundle H)^{\invh- \eta}
}{
(\deg_\cnbundle V)^{\invh+ \eta}
}.$$

If the set of points of $V$ of height at most $\frac{1}{2}\theta$ is empty then $\mu(V)\ge \frac{1}{2}\theta$.
If not, choose a point $\xi \in V$ such that $h(\xi)\le \frac{1}{2}\theta$.
Then $V-\xi \subset H$. Translations preserve transversality and degrees. By  theorem \ref{canoni} for $V-\xi$, $$\mu(V-\xi)\ge \theta.$$
If $x\in V$ and $h(x)\le \frac{1}{2}\theta$, then $x -\xi\in V-\xi$ and $h(x-\xi)\le h(x)+h(\xi)\le \theta\le\mu(V-\xi) $. This shows that 
$$\mu(V)\ge  \frac{1}{2}\theta.$$
\end{proof}

\section{A Matrix Transformation}
\label{sette}
The method used to prove theorem \ref{due} works for matrices which have certain minors different from zero. Such a condition is an open condition. Therefore, any matrix can be approximate with a matrix satisfying such a condition via a `small' rotation. As rational numbers  are dense in the reals, one can assume that the rotation is rational. We want to ensure that the rotation is integral and that the absolute value of the entries is controlled by an absolute constant.
We explicitly
construct the  transformation. As usual, the most
complicated part is to control the size of its entries (see
proposition \ref{cool}).

Let $1\le  n\le N$ be integers. We denote by $\idi_N$ the identity matrix of size $N$.
For a matrix $\psi\in \mat_{n\times N}(\re)$ we denote by $\psi_i$ the $i$-th column of $\psi$.
For a  multi-index  $I=(i_1,\dots,i_{n}) $ with $i_j\in\{1,\dots,N\}$, we define the associated minor
 $$M_I(\psi)=\det (\psi_{i_1}\dots \psi_{i_{n}}).$$

For $1\le i,j\le N$ and $\lambda\in \re$, we denote by $\mathbb{E}_{i,j}(\lambda)$
 the matrix  such that  the  entry at the $ i$ row and $j$ 
 column is equal to $\idi_{ij}+\lambda$ and all other 
 entries are equal to the corresponding entry of the identity. 
 Note that, for a matrix $X$, the multiplication $$ X \mathbb{E}_{i,j}(\lambda)=(X_1,\dots, X_j+\lambda X_i, \dots, X_N).$$

Let $\rho$ be a subset of $\{1,\dots,N\}$.
We say that $\rho \in I$ if all elements of $\rho$ are coordinates of $I$, and $\rho \not\in I$ if such a condition is not satisfied.

\begin{lem}
\label{casino}
Let $\phi \in \mat_{n\times N}(\rend)$ be a matrix. Assume that all the $n\times n$ minors of the matrix consisting of  the first $N-1$ columns of $\phi$ are non zero. Then, there exist integers $\lambda_1,\dots \lambda_n$ and  a permutation $ i_1,\dots ,i_n$ of $1,\dots, n$ such that:

\begin{enumerate}
\item  $ |\lambda_i| \le \frac{1}{N}{N\choose n}$, 

\item  For $\Lambda={\mathbb{E}}_{i_1,N}(\lambda_1)\dots \mathbb{E}_{i_n,N}(\lambda_n)$, 
all the $n\times n$ minors of $\phi\Lambda$ are non zero.
\end{enumerate} 

\end{lem}
\begin{proof}To prove the lemma  is equivalent to prove the following claim.\\

{\bf Claim}

For $0\le r\le n$,  there exists  a set $\rho_r\subset \{1,\dots,n,N\}$,  an index $ i_r\in \{1,\dots,n,N\}$ and  an integer $\lambda_r $, such that

\begin{enumerate}
\item $ |\lambda_r|\le \frac{1}{N}{N\choose n}$,
\item $|\rho_r|=r+1$ , 

\item Define \begin{equation*}
\begin{split}
\Lambda^r&=\mathbb{E}_{i_0,N}(0){\mathbb{E}}_{i_1,N}(\lambda_1)\dots {\mathbb{E}}_{i_r,N}(\lambda_r),\\
\phi^r&=\phi \Lambda^r.
\end{split}
\end{equation*} 
Then, for any multi-index $I$ such that  $\rho_r\not\in I$, $$M_I( \phi^r)\not=0.$$

\end{enumerate}

First we clarify that the claim for $r=n$ proves the  lemma. Indeed, the cardinality of $\rho_{n}$ is $n+1$, so no index $I$  contains $\rho_{n}$.  Then all the $n\times n$ minors of $\phi \Lambda^r$ are non zero. In addition $\Lambda=\Lambda^r$, because ${\mathbb{E}}_{l,k}(0)=\idi_N$.\\

Then, we prove the claim by induction on $r$.\\

Let $r=0$.
Define  $\rho_0=\{N\}$, $i_0=N$  and $\lambda_0=0$. Note that ${\mathbb{E}}_{N,N}(0)=\idi_N$ and $\phi^1=\phi$.
By assumption all $n\times n$ minors of the first $N-1$ columns of $\phi$  are non zero. Equivalently, If $N\not\in I$ then $M_I( \phi^1)\not=0$.
 So the claim is satisfied for $r=0$\\

Let $r\ge1$. Suppose to have proven the claim for $r-1$, we prove it for $r$.
If all minors of $ \phi^{r-1}$ are non zero, define $\lambda_r=0$, choose any element $i_r \not \in  \rho_{r-1}$ with $1\le i_r \le n$ and define  $\rho_r=\rho_{r-1}\cup i_r$.  Otherwise, 
choose a multi-index $I_r$ such that   $M_{I_r}( \phi^{r-1})=0$. By claim iii. for $r-1$, $N\in \rho_{r-1}\in I_r $. Decompose $I_r=(I_0,N)$ where $I_0$ is a multi-index which  has $n-1$ entries. Then, for a question of cardinality, there exists $1\le  i_r\le n$ and  $ i_r\not \in I_0$. So $ i_r\not\in \rho_{r-1}$.   Define  $\rho_r=\rho_{r-1}\cup i_r$. Then claim ii. is satisfied for $r$.

Let $ S_r$ be the set of values $-\frac{M_I( \phi^{r-1})}{M_{J}( \phi^{r-1})}$, for $I$ ranging over all indeces  $I=(I_1,N)$ with $ i_r\not\in I_1$ and $J=(I_1, i_r)$. Note that $M_{J}( \phi^{r-1}) =M_{J}(\phi)\not=0$, by assumption.
The cardinality of $ S_r$ is at most $ \frac{1}{N}{N\choose n}$.
For a question of cardinality, there exists an integer $\lambda_r\not\in S_r$ such that $|\lambda_r|\le \frac{1}{N}{N\choose n}$. Note that $\lambda_r\not=0$ because $M_{I_r}( \phi^{r-1})=0$ is a value in $ S_r$.

  Using the linearity of the determinant on the rows, we show claim iii. for $r$. Suppose that $\rho_r\not\in I$. 
\begin{itemize}

\item[-] If  $N\not\in I$ or $ i_r\in I$, then $\rho_{r-1}\not\in I$ and     $M_I( \phi^r)=M_I( \phi^{r-1})\not=0$, because of claim  iii. for $r-1$,
\item[-] If  $N\in I$  and $ i_r\not\in I$,  then $I  =(I_1,N)$ and $M_I( \phi^r)=M_I( \phi^{r-1})+\lambda_r M_{(I_1, i_r)}( \phi^{r-1})\not=0$
because $\lambda_r\not\in  S_r$.

\end{itemize}

\end{proof}
\begin{propo}
\label{cool}

Let $\psi\in \mat_{n\times N}(\rend)$ be a matrix of  rank $n$.  There exists a permutation matrix $J$ and 
 an  upper triangular integral matrix  $T\in SL_{N}(\ze)$ such that 
\begin{enumerate}
\item $T=\begin{pmatrix}\idi_n&X\\ 0&\idi_{N-n}\end{pmatrix} $ and $|X_{ij}| \le \frac{1}{N}{N\choose n}$, 
\item All the $n\times n$ minors of $\psi JT$ are non zero.
\end{enumerate} 
\end{propo}
\begin{proof}

The rank of $\psi $ is $n$. Then, up to a  permutation of the columns given by a matrix $J$, we can  assume that the first $n$ columns of $\psi $ have rank $n$.

We proceed by induction on $N$. The basis of the induction is $n$.

For $N=n$ the proposition is clearly satisfied with $T=\idi_N$.

Let $N>n$. Suppose the proposition  holds for $ N-1$, we show that it holds for $N$.

Let $\psi \in \mat_{n\times N}(\rend)$  be such that the first $n$ columns have rank $n$. Recall that $\psi_i$ is the $i$-th column of $\psi$.
Define $\phi=(\psi_1,\dots,\psi_{N-1})$. By inductive hypothesis 
there exists an  upper triangular integral matrix  $R\in SL_{N-1}(\ze)$ such that 
\begin{enumerate}

\item $R=\begin{pmatrix}\idi_n &Y\\ 0&\idi_{N-n-1}\end{pmatrix}$ and $|Y_{ij}| \le \frac{1}{N-1}{N-1\choose n}$, 
\item All the $n\times n$ minors of $\phi R$ are different from zero.

\end{enumerate} 
Define $T '=\begin{pmatrix}R &0\\0&1\end{pmatrix}$. Then $\psi '=\psi  T '=(\phi R|\psi_N)$ is such that all minors of the first $N-1$ columns are non-zero.
Apply lemma \ref{casino}  to $\psi '$.  Then for  $ |\lambda_i| \le \frac{1}{N}{N\choose n}$ and 
 $\Lambda={\mathbb{E}}_{i_1,N}(\lambda_1)\dots {\mathbb{E}}_{i_n,N}(\lambda_n)$, all the $n\times n$ minors of $\psi '\Lambda$ are non zero.  Define $T=T '\Lambda$. Note that $T=\begin{pmatrix}\idi_n &X\\ 0&\idi_{N-n}\end{pmatrix}$ where $X=(Y|\lambda) $ and $\lambda$ is the column vector given by a permutation of $(\lambda_{1},\dots,\lambda_{n})$. This proves the  proposition  for $N$. 
\end{proof}

\section{Appendix I: A conjectural implication}
F. Amoroso, S. David and P. Philippon related the essential minimum of a transverse variety  with the
relative obstruction index.
Let  $H$ be the translate of an abelian subvariety of $E^N$ containing $V$.
   The relative obstruction index is
$$\omega_\elle(V,H) =\min_{V\subset Z}\deg _\elle Z,$$
for $Z$ varying over all divisors of $H$ containing $V$.

The following conjecture can be deduced form work of Amoroso, David and Philippon.

\begin{con}
\label{funtoriale}
Let $H$ be the  translate of an abelian subvariety of $\gum^N$ of dimension $n$.
Let $V$ be transverse in $H$. 
For any  symmetric ample line bundle $\elle$ on $\gum^N$, there exists a positive constant $ c_1$ depending on $E^N$ and $\elle$ such that 
$$\mu_{\elle}(V)\ge c_1 \frac{\deg_{\elle} H}{\omega_{\elle}(V,H)}.$$
\end{con}

In view of our theorem \ref{due} we can suggest

\begin{con}
\label{bofu}
Let $H$ be the  translate of an abelian subvariety of $\gum^N$ of dimension $n$.
Let $V$ be a $d$-dimensional variety transverse in $H$. 
Then, for any polarization $\elle$ on $E^N$, there exists a positive constant $ c'_1$ depending on $E^N$ and $\elle$ such that 
$$\mu_{\elle}(V)\ge c'_1 \left(\frac{\deg_{\elle}H}{\deg_{\elle}V}\right)^{\frac{1}{\codv}}.$$
\end{con}
 
We conclude our work clarifying the relation, pointed out by Philippon, between the conjecture \ref{funtoriale}  and our conjecture \ref{bofu}. 
\begin{propo}
\label{phil}
Conjecture \ref{funtoriale} implies Conjecture \ref{bofu}.
\end{propo}
\begin{proof}
We denote by $\ll$ an inequality up to a multiplicative constant, depending only on irrelevant parameters.
Let $H$ be the translate of an abelian subvariety of dimension $n$ and let $\elle$ be any line bundle on $H$. 
Let $V$ be transverse in $H$ of dimension $d$. We shall prove that
$$\frac{ \omega_\elle(V,H)}{ \deg_\elle H}\ll \left(\frac{\deg_\elle V}{\deg_\elle H}\right)^{\frac{1}{n-d}}.$$

The main result of M. Chardin~\cite{Cha 1988}  gives the following upper bound for the Hilbert function
$$\mathcal{H}_\elle(V,\nu)\ll \frac{\deg_\elle V}{d !}\nu^{d}.$$
The remark 1 of \cite{Cha 1988}, implies the lower bound for the Hilbert function
$$\mathcal{H}_\elle(H,\nu)\gg \frac{\deg_\elle H}{n !}\nu^{n},$$
for $\nu \gg {\deg_\elle H}$.

Choose $\nu$ minimal so that $\mathcal{H}_\elle(H,\nu)\le \mathcal{H}_\elle(V,\nu)$ and $\nu \gg{\deg_\elle H}$. Then
$$\nu\ll \left(\frac{n !}{d!} \frac{\deg_\elle V}{\deg_\elle H}\right)^{\frac{1}{n-d}}.$$

Equivalently, there exists a hyperplane of $\mathbb{P}^n$ containing $V$ but not $H$ of degree $\le \nu$.
By Bezout's theorem, this hypersurface defines a divisor of $H$ containing $V$ of degree $\le \nu \deg_\elle H$.
It follows
$$ \omega_\elle(V,H)\le \nu \deg_\elle H\ll \deg_\elle H \left(\frac{\deg_\elle V}{\deg_\elle H}\right)^{\frac{1}{n-d}}.$$
This concludes the proof.

\end{proof}

Finally we remark that in  \cite{sipa}, David and Philippon manage to obtain a lower bound for the essential minimum of a transverse variety in a power of an elliptic curve where $h(E)$ is at the numerator. Unfortunately they lost in the dependens on $\deg_\cnbundle V$. This is not strong enough to apply the method presented in this work and to extend their result to other polarizations. However, they announce a strong conjecture, \cite{sipa}  conjecture 1.5 ii. Using our method,  we can conclude that if \cite{sipa} conjecture 1.5 ii. holds for $\cnbundle$,  then it holds for  the restrictions of the standard line bundle to translates of abelian subvarieties.

\vskip1cm


\newpage



\begin{thebibliography}{mic}
\bibitem{fra} F. Amoroso and S. David, {\em Minoration de la hauteur normalis\'ee dans un tore.} J. Inst. Math. Jussieu, 2 (2003), no.3, 335-381.


\bibitem{fraio} F. Amoroso and E. Viada, {\em  Small points on
    subvarieties of tori,}   to appear in Duke Math.l Journal, (2009), 28 pages.

\bibitem{B} M. Baker and J. Silverman, {\em A lower bound for the canonical height on abelian varieties over abelian extensions}, Mathematical Research Letters, 11 (2004) 377-396.


\bibitem{BZ} E. Bombieri and U. Zannier, {\em Heights of Algebraic
    Points on Subvarieties of Abelian Varieties}, Ann. Scuola
     Normale Sup. Pisa, cl. scienze, Serie IV 23 (1996):
     779-792.

\bibitem{Cha 1988}  
M.~Chardin, {\em Une majoration de la fonction de
Hilbert et ses cons\'equences pour l'interpolation alg\'ebrique}, 
Bulletin de la Soci\'et\'e Math\'ematique de France,
{\bf 117}, p. 305-318 (1988).




\bibitem{sihi}  S. David.  and M. Hindry  {\em Minorations de la  hauteur de N\'eron-Tate sur les  vari\'et\'es abeliennes de type C.M..}  J. Reine Angew. Math. 529 (2000), 1-74.
  
\bibitem{sipa1}  S. David.  and P. Philippon, {\em Minorations des hauteurs normalis\'ees des sous-vari\'et\'es de vari\'et\'es abeliennes,} In Number Theory (Tiruchiraparalli, 1996), V. K. Murty and M. Waldschmidt ed. Contemporary Math. 210 (1998), p. 333-364.


\bibitem{sipacommentari}  S. David.  and P. Philippon, {\em Minorations des hauteurs normalis\'ees des sous-vari\'et\'es de vari\'et\'es abeliennes II,} Comment. Math. Helv. 77 (2002), no. 4, p. 639-700.

\bibitem{sipa}  S. David.  and P. Philippon, {\em Minorations des hauteurs normalisées des sous-variétés des puissances des courbes elliptiques},  Int. Math. Res. Pap.   2007,  no. 3, Art. ID rpm006, 113 pp


\bibitem{Gal}  A. Galateau, {\em  Une minoration du minimum
    essentielle sur les varie\'et\'es ab\'eliennes.}
  http://arxiv.org/PS\_cache/arxiv/pdf/0807/0807.0171v1.pdf\,\,\,\,, to appear in  Comment. Math. Helv. (2009).



\bibitem{Hin}  M. Hindry, {\em  Autour d'une conjecture de Serge Lang}
 Invent. Math. 94, 1988, 575-603.



\bibitem{l-b} H. Lange and Ch. Birkenhake, {\em Complex Abelian Varieties}. Springer-Verlag, Berlin-Heidelberg, 1992.

\bibitem{Masser} D. Masser, {\em Small values of heights on families of
    abelian varieties}, in `Diophantine approximation and
  transcendence theory' (Bonn,1985) G. W\"usthoz editor, Lecture Notes
  in Math 1290 (1987), 109-148.

\bibitem{MW}  D. Masser and G. W\"ustholz, {\em  Endomorphism estimates for abelian varieties.}
Math. Z. 215, 1994, p. 641-653.
\bibitem{Masserwustholz}  D. Masser and G. W\"ustholz, {\em  Periods and minimal abelian dubvarieties.}
Annals of math., 137 (1993), 407-458.

\bibitem{sombrapa} P. Patrice and  M. Sombra, {\em  Minimum essentiel
    et degr\'es d'obstruction des translat\'es de sous-tores},  Acta
  Arith.  133  (2008),  no. 1, 1--24.
  
\bibitem{bertrand} N. Ratazzi and E. Ullmo  {\em Galois plus \'Equidistribution=Manin-Mumford} in Acts of the summer school "Arithmetic Geometry" of
Goettingen   2006.

 
\bibitem{raynaud}
M. Raynaud,  {\em Courbes sur une vari\'et\'e  ab\'elienne et points
  de torsion},   Invent. Math. 71, 1983, no. 1, 207--233.

\bibitem{Sil} J. Silverman, {\em Lower Bounds for the Height Functions},
  Duke Math. Journal 51 (1984):395-403.

\bibitem{Sc}
W. M. Schmidt. ``Heights of points on subvarieties of $\gum^n$". In
``Number Theory 93-94", S. David editor, 157--187. London Math. Soc.
Ser., volume 235, Cambridge University Press, 1996.

  \bibitem{ulmo} E. Ullmo, {\em Positivit\'e et discr\'etion des points alg\'ebriques des courbes.} Ann. of Math. 147 (1998), no. 1,  167-179.

\bibitem{ant}
E. Viada, {\em The intersection of a curve with a union of translated codimension $2$ subgroups in a power of an elliptic curve}. Algebra and Number Theory, Vol. 2, N. 3, (2008), 248--298.


\bibitem{irmn} E. Viada, {\em Non-dense subsets of varieties in a
    power of an elliptic curve,} Int. Math. Research
  Notices, (2009) 34 pages.

\bibitem{ijnt} E. Viada 
 {\em Lower bounds for the normalized height and  non-dense subsets of subvarieties in an abelian variety, } to appear in Int. Journal of Number Theory, (2009) 25 pages.

\bibitem{zhang} S. Zhang, {\em Equidistribution of small points on abelian varieties,} Ann. of Math., 147 (1998), no. 1, 159-165.
\end{thebibliography}
\end{document}